\documentclass[12pt, reqno]{amsart}
\usepackage{amsmath, amsthm, amscd, amsfonts, amssymb, graphicx, color}
\usepackage{enumerate}
\usepackage[bookmarksnumbered, colorlinks, plainpages]{hyperref}
\hypersetup{colorlinks=true,linkcolor=red, anchorcolor=green, citecolor=cyan, urlcolor=red, filecolor=magenta, pdftoolbar=true}

\newtheorem{theorem}{Theorem}[section]
\newtheorem{lemma}[theorem]{Lemma}

\theoremstyle{definition}

\newtheorem{question}[theorem]{Question}
\theoremstyle{remark}
\newtheorem{remark}[theorem]{Remark}
\numberwithin{equation}{section}

\begin{document}

\setcounter{page}{1}

\title[Hyponormal and subnormal dual Toeplitz operators]{Hyponormal and subnormal dual Toeplitz operators}

\author[Linghui Kong, Yongning Li, Da Sun \MakeLowercase{and} Dechao Zheng]{Linghui Kong, Yongning Li, Da Sun \MakeLowercase{and} Dechao Zheng}

\address{School of Mathematics and Statistics, Northeast Normal University, Changchun, 130024, China.}
\email{konglh555@nenu.edu.cn}

\address{School of Mathematics and Statistics, Chongqing Technology and Business University, Chongqing, 400067, China.}
\email{yongningli@ctbu.edu.cn}

\address{School of Mathematics and Statistics, Northeast Normal University, Changchun, 130024, China.}
\email{sunda@nenu.edu.cn}

\address{Department of Mathematics, Vanderbilt University, Nashville, Tennessee, 37240, United States.}
\email{dechao.zheng@vanderbilt.edu}


\subjclass[2010]{Primary 47B35; Secondary 47B20.}

\keywords{dual Toeplitz operator, hyponormality, subnormality, Bergman space, $\varTheta$-class.}


\begin{abstract}
In this paper, we completely characterize hyponormal dual Toeplitz operators with general bounded measurable symbols by using the Douglas  theorem in \cite{Douglas1966}.
  As an application, the hyponormality of dual Toeplitz operators with radial symbols and polynomial symbols in $z$ and $\bar{z}$ is characterized concretely.
  Moreover, we obtain complete characterizations for both subnormal dual Toeplitz operators with harmonic polynomial symbols and $\varTheta$-class dual Toeplitz operators with bounded harmonic symbols.
\end{abstract} \maketitle

\section{Introduction}\label{sec1}

Let $L^2_a$ denote the Bergman space which is  the Hilbert space of analytic functions on the unit
disk $\mathbb{D}$ that are square integrable with respect to normalized area measure $dA$.
Then  $L^{2}(\mathbb{D}, dA)$ has the following decomposition:
$$L^{2}(\mathbb{D}, dA)= L^2_a\oplus (L^2_a)^{\perp}.$$
Let $P$ be the orthogonal projection from $L^{2}(\mathbb{D}, dA)$ onto the Bergman space $L^2_a$.
For $\phi\in L^{\infty}(\mathbb{D})$, under the above decomposition, the multiplication operator  $M_\phi$  defined by
$$M_\phi f= \phi f,$$
for $f\in L^{2}(\mathbb{D}, dA)$ has the following matrix representation:
\begin{eqnarray}\label{matixrep}
	M_{\phi}=\left(\begin{array}{cc }
		T_{\phi} & H_{\bar{\phi}}^*\\
		H_{\phi} & S_{\phi}
			\end{array}\right )
\end{eqnarray}
where $T_{\phi}:L^2_a\rightarrow L^2_a$ is  the Toeplitz operator with symbol $\phi$ defined by
$$T_{\phi}h =P(\phi h);$$
and $H_\phi: L^{2}_a\rightarrow (L^2_a)^{\perp}$ is the Hankel operator with symbol $\phi$ defined by
$$H_{\phi}h =(I-P)(\phi h);$$
$S_{\phi}:(L^{2}_a)^{\perp}\rightarrow (L^2_a)^{\perp}$ is the dual Toeplitz operator  with symbol $\phi$  defined by
$$S_{\phi}g=(I-P)(\phi g).$$
Zhu's book \cite{Zhu2007} contains more details on the Toeplitz operators and Hankel operators on the Bergman space $L^2_a$. The dual Toeplitz operator is introduced in \cite{SZ1999}. In \cite{SZ2002} it was shown that dual Toeplitz operators share many algebraic and spectral properties  with Toeplitz operators on the Hardy space of the unit circle. In particular, $S_{\phi}$ with its symbol $\phi\in L^2(\mathbb{D}, dA)$ is bounded if and only if $\phi$ is in $L^{\infty} (\mathbb{D}).$

Suppose  $A$ is a bounded linear operator on a Hilbert space $\mathcal H$. $A$ is said to be normal if $A^*A=AA^*$; $A$ is said to be hyponormal if $$A^*A-AA^*\geq 0 ;$$
$A$ is said to be subnormal if $A$ is the restriction to an invariant subspace of a normal operator.
In \cite{SZ2002} it was shown that $S_\phi$ is normal if and only if $\phi =\alpha +\beta \psi$ for two constants $\alpha$, $\beta$ and a real valued function $\psi\in L^{\infty}(\mathbb{D}),$ which is analogous to the Brown and Halmos theorem \cite{BrH} about normal Toeplitz operators on the Hardy space.
In this paper we will study hyponormal dual Toeplitz operators.
Using the following Douglas theorem \cite{Douglas1966}, we obtain a complete characterization on hyponormal dual Toeplitz operators.

\begin{theorem}(\cite{Douglas1966})\label{Thm00}
	Let $A$ and $B$ be bounded operators on the Hilbert space $\mathcal H$.
	Then $AA^*\leq  BB^*$  if and only if
	there exists a bounded operator $C$ on $\mathcal H$ so that $A=BC$.
\end{theorem}

Let $H^{\infty}(\mathbb{D})$ be the algebra of bounded analytic functions on the unit disk and $H^{\infty}$ the algebra of functions on the unit circle $\partial \mathbb{D}$ whose harmonic extensions are bounded analytic on the unit disk.
Our main result is the following theorem.

\begin{theorem}\label{thm1}
	Suppose $\phi$ is a function in $L^{\infty}(\mathbb{D})$.
	Then $S_{\phi}$ is hyponormal if and only if one of the following two conditions is satisfied:
	\begin{enumerate}[(i)]
		\item  there exist  a function $f\in  H^\infty (\mathbb{D})$ and a function $u\in H^\infty (\mathbb{D})$ with      $\|u\|_{\infty}\leq 1$ and $|u(0)|<1$ such that  $$\phi (z)=\frac{\overline{u(z)}f(z)+\overline{f(z)}}{1-|u (z)|^2}$$
		for $z\in \mathbb{D}$;
		\item  there is a unimodular constant $c$ such that  $\overline{\phi(z)}-c\phi(z) $ is a constant for $z\in \mathbb{D}$.
	\end{enumerate}
\end{theorem}

The above theorem is analogous to the Nakazi and Takahashi theorem \cite{Nakazi1993} on hyponormal Toeplitz operators on the Hardy space.

\begin{theorem}(\cite{Nakazi1993})
	Let $\phi$ be a nonzero function in $L^\infty(\partial{\mathbb{D}})$. $T_{\phi}$ is hyponormal on the Hardy space $H^2$ if and only if
	$$\phi =\chi_{E}\left(k^{\frac{1}{2}}u+\frac{1}{2}g\right)+\chi_{E^c}\frac{k\bar{g}+g}{1-|k|^2},$$
	where $k$ and $g$ are functions in $H^{\infty}$ with $\|k\|_{\infty}\leq 1$, $E=\{e^{i\theta}\in \partial \mathbb{D}:|k(e^{i\theta})|=1\}$, $k\bar{g}=-g$ a.e on $E$ and $u$ is a real-valued function in $L^{\infty}(\partial{\mathbb{D}})$.
\end{theorem}

 Clearly,  Condition  (i)  or (ii) in Theorem \ref{thm1} is equivalent to the following condition: there is a function $u$ in $H^\infty (\mathbb{D})$ with $\|u\|_{\infty}\leq 1 $ such that
$$\bar{\phi}-u\phi\in H^\infty (\mathbb{D}).$$
So Theorem \ref{thm1} can be reformulated as the following form.

\begin{theorem}\label{theorem3.2}
Let $\phi$ be a function in $L^{\infty}(\mathbb{D})$.
Then $S_{\phi}$ is hyponormal if and only if
the set $$\left\{u\in H^\infty (\mathbb{D}):\|u\|_{\infty}\leq1~\text{and}~\bar{\phi}-u\phi\in H^\infty (\mathbb{D})\right\}$$ is not empty.
\end{theorem}

 The above theorem  bears a resemblance to Cowen's theorem on hyponormal Toeplitz operators on the Hardy space $H^2$   (\cite{Cowen1988}, \cite{Nakazi1993}).
\begin{theorem}(\cite{Cowen1988}, \cite{Nakazi1993})\label{cowen}
	Let $\phi$ be in $L^\infty(\partial{\mathbb{D}})$.
	Then the Toeplitz operator $T_{\phi}$ on the Hardy space $H^2$ is hyponormal if and only if
	the set $$\left\{u\in H^\infty  :\|u\|_{\infty}\leq1~\text{and}~\phi-u\bar{\phi}\in H^\infty  \right\}$$ is not empty.
\end{theorem}

On the Hardy space, by using Schur's function, Zhu \cite{Zhu1995} characterized hyponormal Toeplitz operators with trigonometric symbols. Subsequently, Farenick and Lee \cite{Far1996} derived a new criterion in terms of the Fourier coefficients of the trigonometric polynomial.
Recently, Toeplitz matrices were used to characterize hyponormal Toeplitz operators with trigonometric symbols in \cite{DZZ}.
If the dual Toeplitz operator $S_\phi$ is hypnomoral but not normal, Condition (i) in Theorem \ref{thm1} must hold and hence implies that $\phi$ is a real analytic function on the unit disk , i.e., $\phi(z)$ has a power series expansion in $z^n\bar{z}^m$:
$$\phi (z)=\sum_{n=0}^{\infty}\sum_{m=0}^\infty a_{n, m}z^n\bar{z}^m,$$
for $z\in \mathbb D$ as
\begin{equation}\label{realAna}
	\phi (z) =\big[\overline{u(z)} f(z)+\overline{f(z)}\big]\left(\sum_{n=0}^{\infty}|u(z)|^{2n}\right), 	
\end{equation}
and both $f$ and $u$ are analytic on the unit disk and $|u(z)|<1$ for each $z$ in the unit disk.
Thus $\phi (z)$ is a function in $C^{\infty} (\mathbb{D}).$ This indicates that many hyponormal dual Toeplitz operators $S_\phi$ with special symbols must be normal if their symbols are too``bad". The following theorem confirms such phenomena.

\begin{theorem}\label{thm5}
	Let $\phi$ be a bounded and radial function on $\mathbb{D}$. If $S_{\phi}$ is hyponormal, then  $S_{\phi}$ is normal.
\end{theorem}

When we study hyponormality of operators,   the dual Toeplitz operators are more friendly than Toeplitz operators on the Bergman space $L^2_a$.
On the Bergman space $L_a^2$, the hyponormality of Toeplitz operators seems more complicated, only partial sufficient or necessary conditions were found except for polynomial symbols of specific forms. Sadraoui \cite{Sadraoui1992} established a connection between the hyponormality of Toeplitz operators and the boundary behavior of their symbols: Let $f,g$ be bounded analytic and
$f'\in H^2$. If $T_{f+\bar{g}}$ is hyponormal, then $g'\in H^2$ and $|g'|\leq|f'|$ almost everywhere on the boundary $\partial\mathbb{D}$.
Additionally, a characterization of hyponormal Toeplitz operators of the form $T_{z^n+\alpha\bar{z}^m}$ was given there.
Adopting Berezin transform as a central tool, Ahern and \v{C}u\v{c}kovi\'{c} \cite{Ahern1996} showed that if $T_{f+\bar{g}}$ is hyponormal, then $\mathcal{B}u\geq u$, where $\mathcal{B}$ denotes the Berezin transform and $u=|f|^2-|g|^2$. They also applied this result to generalize Sadraoui's work.

Theorem \ref{thm1} also immediately gives the characterization of hyponormal dual Toeplitz operators with bounded harmonic symbol obtained first in \cite{Ding2020}.
Let $BMOA$ denote the space of analytic functions in the Hardy space $H^2$ whose boundary values have bounded mean oscillation. Then it was pointed out in \cite{AxZ} that  $$P(h^{\infty}(\mathbb{D}))=BMOA,$$
  where  $h^{\infty}(\mathbb{D})$ denotes the space of bounded harmonic functions on the unit disk.

\begin{theorem}\label{Cor00}
	Suppose  $\phi $ is a bounded and  harmonic function on the unit disk.  Then
	$S_{\phi}$ is hyponormal if and only if either there are a unimodular constant $\xi$, a constant $\beta$ and a function $g$ in $BMOA$ with $\Re g$ in $L^{\infty}(\mathbb{D})$ such that
	$$\phi =\xi (g+\bar{g})+\beta$$
	 or there are constants $\alpha$ and $\beta$ with $|\alpha|<1$ and a function $g$ in $H^{\infty}(\mathbb{D})$ such that
	$$\phi=\alpha g+\bar{g}+\beta.$$
\end{theorem}

In the proof of Theorem \ref{Cor00}, Ding, Wu and Zhao \cite{Ding2020} used the commutant characterization of the Bergman backward shift in \cite{Bourdon2000}.  We will give a different proof by using Theorem \ref{thm1}. The result in Theorem \ref{Cor00} was later generalized to block dual Toeplitz operators with matrix-valued bounded harmonic symbols in \cite{Fu2025}. Besides, on the harmonic Bergman space, Wang and Zhao \cite{Wang2023} discovered the equivalence between the hyponormality and normality of dual Toeplitz operators.

   For general polynomial symbols $\phi$ of $z$ and $\bar{z}$, it is unknown when $T_{\phi}$ is hyponormal on the Bergman space. In 2005, Hwang \cite{Hwang2005} studied the Toeplitz operator $T_{a_mz^m+a_nz^n+b_m\bar{z}^m+b_n\bar{z}^n}$ on weighted Bergman spaces, where $0<m<n$ and $a_m\bar{a}_n=\bar{b}_mb_n$. Hwang obtained a sufficient and necessary condition involving the coefficients $a_m, a_n, b_m, b_n$ for the hyponormality. In 2009, Lu and Shi \cite{Shi2009} revisited this type of operators, proving that the necessary condition given in \cite{Sadraoui1992} becomes sufficient under the constraint $|a_n|\leq |b_n|$. In the same year, Lu and Liu \cite{Liu2009} investigated hyponormal Toeplitz operators with quasi-homogeneous symbols.
     Fleeman and Liaw \cite{FL2019} showed that
    if  $\phi(z)=a_1z^{m_1}\bar{z}^{n_1}+\cdots+a_kz^{m_k}\bar{z}^{n_k}$ with $m_1-n_1=\cdots=m_k-n_k=\delta>0$, and $a_i$ all
    		lie in the same quarter plane for $1\leq i\leq k$, then $T_{\phi}$ is hyponormal.
  Using Theorem \ref{thm1} we obtain a characterization of hyponormal dual Toeplitz operators with polynomial symbols of $z$ and $\bar{z}$ in terms of their symbols in the following theorem.

    \begin{theorem}\label{poly}   Suppose $\phi$ is a polynomial of $z$ and $\bar{z}$:
   	$$\phi (z)=\sum_{i=0}^{n}\sum_{j=0}^n a_{ij}z^i{\bar{z}}^j.$$
   	The dual Toeplitz operator $S_{\phi}$ is hyponormal if and only if one of the following two conditions holds:
   	
   	\begin{enumerate}[(i)]
   		\item  the dual Toeplitz operator $S_{\phi}$ is normal and  there is a unimodular constant $c$ such that
   		$$\overline{ a_{ji}}=ca_{ij}$$
   		for     $(i, j)\in \{0, \cdots, n\}\times \{0, \cdots, n\}$ with $(i, j)\neq (0, 0)$;
   		\item  there exists a function $u\in H^{\infty}(\mathbb{D})$ with $\|u\|_{\infty}\leq 1$ and $|u(0)|<1$ such that  for $z\in \mathbb{D}, $
   		\begin{eqnarray} \label{uexp}
   			u(z)\left({\sum_{j=0}^n{a_{jm}}z^j}\right)&=& {\sum_{j=0}^n\overline{a_{mj}}z^j}
   		\end{eqnarray}
   		for $m=1, \cdots , n.$
   	\end{enumerate}
   \end{theorem}

On the Hardy space $H^2$,  in \cite{Halmos1970}, Halmos raised the question
``Is every subnormal Toeplitz operator either normal or analytic?"
In \cite{Ito1972}, Ito and Wong proved that if $\phi$ is a polynomial of $\chi$ and $\bar{\chi}$, where $\chi$
is an inner function, then $T_\phi$ is subnormal if and only if it is normal or
analytic.
Subsequently, Amemiya, Ito and Wong \cite{Ito1975} proved via similar arguments that Halmos's question admits a positive answer for quasinormal Toeplitz operators.
Abrahamse \cite{Abrahamse1976} gave a very general sufficient condition for an affirmative answer to Halmos's
question: if $\phi$ or $\bar{\phi}$ is of bounded type
and $T_{\phi}$ is subnormal, then $T_{\phi}$ is normal or analytic.
Surprisingly,  Cowen and Long \cite{Cowen1984}, \cite{Cowen1988} showed that the
answer to Halmos's question is negative and gave examples that $T_{\psi +\lambda\overline{\psi}}$ is subnormal but is neither normal nor analytic for some $\psi$ in $H^{\infty}$. In  \cite{S1984} and \cite{S19851}, Sun answered Abrahamse's question that when a Toeplitz operator is unitarily equivalent to a weighted shift, and the answer also yields counterexamples to Halmos's question.  In fact, the characterization of $\varTheta$-class Toeplitz operators in \cite{S1985} also presents counterexamples to Halmos's question. Cowen's survey paper \cite{CownNN} contains more details on hyponormal and subnormal operators.

A dual Toeplitz operator $S_\phi$ is called co-analytic if $\bar{\phi}$ is in $H^{\infty}(\mathbb{D})$. A co-analytic dual Toeplitz operator $S_{\phi}$ is the restriction of the multiplication operator $M_{\phi}$ on $(L^2_a)^{\perp}$ and hence is subnormal as $M_{\phi}$ is normal. About dual Toeplitz operators $S_{\phi}$ with bounded harmonic symbols we may ask the following question.
\begin{question}\label{Q2}
	Is every subnormal dual Toeplitz operator either normal or co-analytic?\label{Q1}
\end{question}
 In section 3, we will show the following result.
\begin{theorem}\label{DSNC}
	Let $\phi$ be a harmonic polynomial in $z$ and $\bar{z}$ on the unit disk.
	Then $S_{\phi}$ is subnormal if and only if $S_{\phi}$ is either normal or co-analytic.
\end{theorem}
 In Section 3 we will also characterize $\varTheta$-class dual Toeplitz operators with bounded harmonic symbols and show that every $\varTheta$-class dual Toeplitz operator with a bounded harmonic symbol must be  normal.

\section{Proof of the main results}\label{sec2}

In this section,  using  the Douglas theorem (Theorem \ref{Thm00} )  we will present the proof of Theorem \ref{thm1}.  Then using Theorem \ref{thm1} we will derive Theorems \ref{thm5}, \ref{Cor00}, and \ref{poly}.

First noting that $M_{\phi\psi}=M_{\phi}M_{\psi}$, using the matrix representation (\ref{matixrep}) of the multiplication operator $M_{\phi}$, we have the following  useful algebraic relations between Toeplitz operators, Hankel operators and dual Toeplitz operators:
\begin{eqnarray}
	T_{\phi\psi} &=& T_{\phi}T_{\psi}+H_{\bar{\phi}}^*H_{\psi},\label{TH}\\
	S_{\phi\psi} &=& S_{\phi}S_{\psi}+H_{\phi}H_{\bar{\psi}}^*, \label{SH}\\
	H_{\phi\psi} &=& H_{\phi}T_{\psi}+S_{\phi}H_{\psi}.  \label{HST}
\end{eqnarray}

\begin{proof}[{\bf Proof of Theorem \ref{thm1}}] Assume that $\phi$ is not constant as   it is obvious that the result holds for a constant  $\phi$.

 For the necessity, suppose $S_{\phi}$ is hyponormal. Then
 $$S_{\phi}^*S_{\phi}-S_{\phi}S_{\phi}^*\geq 0.$$
Use identity (\ref{SH}) to get
\begin{eqnarray}
S_{\phi}^*S_{\phi}-S_{\phi}S_{\phi}^*&=&S_{\bar{\phi}}S_{\phi}-S_{\bar{\phi}\phi}+S_{\bar{\phi}\phi}-S_{\phi}S_{\bar{\phi}} \notag\\
&=&H_{\phi}H_{\phi}^*-H_{\bar{\phi}}H_{\bar{\phi}}^* \label{eqnf0}.
\end{eqnarray}
Thus we have
$$H_{\bar{\phi}}H_{\bar{\phi}}^* \leq H_{\phi}H_{\phi}^*.$$
As from the definition of the Hankel operator $H_{\psi}$, for $\psi\in L^{\infty}(\mathbb{D})$,  $H_{\psi}$ equals zero if and only if $\psi$ is analytic on the unit disk, first we note that $\phi$ is not analytic on the unit disk. Otherwise, $H_{\phi}$ must be zero and the above inequality gives that $H_{\bar{\phi}}$ equals zero and hence both $\phi$ and $\bar{\phi}$ are analytic on the unit disk. This implies that $\phi$ is constant. It is a contradiction.

By the Douglas  theorem (Theorem \ref{Thm00}), the above inequality gives that
there exists a bounded operator $B$ on $L^2_a$ with norm $\|B\| \leq 1$ such that
\begin{equation}\label{proof1}
	H_{\bar{\phi}} = H_{\phi}B.
	\end{equation}
Next we will show that $B$ is an analytic Toeplitz operator. To do this,
let $u=B1$. Then $u$ is in the Bergman space  $L^2_a$ and hence it is analytic on the unit disk.
Since $$H_{\bar{\phi}}1=H_{\phi}B1=H_{\phi}u,$$
we have
 $$(I-P)({\bar{\phi}}-\phi u)=0$$  to get that $ \bar{\phi}-\phi u$ is in the Bergman space $L_a^2.$
Let  $f =\bar{\phi}-\phi u$. Thus
\begin{equation*}
	\bar{\phi} =\phi u+f,
\end{equation*} and so for $g\in H^{\infty}(\mathbb{D}),$   we have
\begin{align*}
	H_{\bar{\phi}}g&=(I-P)(\bar{\phi}g)\\
	&=(I-P)[(\phi u+f)g]\\
	&=(I-P)(\phi ug)+(I-P)(fg)\\
	&=(I-P)(\phi ug)\\
	&=H_{\phi}(ug).
\end{align*}
The fourth equality follows from the fact that $fg$ is in the Bergman space $L^2_a$.
On the other hand, (\ref{proof1}) gives that for $g\in H^{\infty}(\mathbb{D}),$
\begin{align*}
	H_{\bar{\phi}}g&=H_{\phi}Bg.
\end{align*}
This implies
$$H_{\phi}(ug)=H_{\phi}Bg,$$
and hence
$$H_{\phi}(ug-Bg)=H_{\phi}(ug)-H_{\phi}Bg=0,$$
for $g\in H^{\infty}(\mathbb{D}).$
 By   the Das theorem \cite{Das1999} that
 $\text{Ker}~ H_{\phi}=\{0\}$ for  a non-analytic bounded function $\phi$ on $\mathbb{D}$, we have
 $$Bg=ug$$
 for all $g\in H^{\infty}(\mathbb{D})$.

 For completeness, we include a proof of the Das theorem \cite{Das1999} here. In fact, if $h$ is in $\text{Ker}~ H_{\phi}$,
 then $$0=H_\phi h =(I-P)(\phi h).$$
 Let $g=P(\phi h)$. If $h$ is not equal to the zero function on the unit disk, the above equation gives
   $$\phi (z)=\frac{g(z)}{h(z)}$$
   holds on $\mathbb{D}$ except for the zero set $\mathcal{Z}(h)$ of $h$ that is countable and the right hand side of the above equation is also analytic on $\mathbb{D}\backslash \mathcal{Z}(h)$ since both $g$ and $h$ are analytic on $\mathbb{D}. $ As $\phi$ is bounded on $\mathbb{D}$, $\phi$ is bounded and analytic on $\mathbb{D}$. This is a contradiction.

 Next using the Berezin transform on the Bergman space we will show that $u$ is in $H^{\infty}(\mathbb{D})$.
Since $\|B\|\leq1$, its Berezin transform
$$|\langle Bk_z,k_z\rangle |\leq \|B\|\leq 1.$$
Noting each normalized reproducing kernel $k_z$ is in $H^{\infty}(\mathbb{D})$, we have
$$ \langle Bk_z,k_z\rangle = \langle uk_z,k_z\rangle =u(z),$$
for $ z\in\mathbb{D}.$ The last equality follows from the property of reproducing kernels. Thus $ u$ is in $H^{\infty}(\mathbb{D})$ with $\|u\|_{\infty}\leq1$. As $H^{\infty}(\mathbb{D})$ is dense in the Bergman space $L^2_a$, we have
$$Bh=uh =T_{u}h,$$
 for $h$ in the Bergman space $L^2_a$ to get that $B$ is equal to the analytic Toeplitz operator $T_u$.

As $\phi$ and $u$ are bounded, we have
\begin{equation}\label{eqnf1}
f(z)=\overline{\phi (z)}-\phi (z) u(z)
\end{equation}
is in $H^{\infty}(\mathbb{D})$.
Take the conjugate on both sides of the above equality to get
\begin{equation}\label{eqnf2}
\overline{f(z)}= {\phi (z)}-\overline{\phi (z)} ~\overline {u(z)}.
\end{equation}
 Multiply both sides of (\ref{eqnf1}) by $\overline{u (z)}$ and then use (\ref{eqnf2})
 to get
 \begin{equation}\label{eqnf3}
    (1-|u(z)|^2)\phi (z)=\overline{u (z)}f(z)+\overline{f (z)}.
 \end{equation}
 If $1-|u(z)|^2=0$ for $z\in  \mathbb D$,  as $u$ is analytic on the unit disk,  by the Maximum Modulus Theorem, then $u$ is a unimodular constant $c$ and the above equality gives
 $$\overline{f (z)}= -\overline{c}f (z),$$
 for $z\in \mathbb D$. This implies that $f$ is constant and hence (\ref{eqnf1}) gives
 $\overline{\phi (z)}-c{\phi (z)}$ is a constant on the unit disk.

 If $u$ is not a unimodular constant, then $\|u\|_{\infty}\leq 1$ and $|u(0)|<1.$ Solving for $\phi (z)$ in Equation (\ref{eqnf3}) gives
 $$\phi (z)=\frac{\overline{u(z)}f(z)+\overline{f(z)}}{1-|u (z)|^2}.$$ This completes the proof of  the necessity.

For the sufficiency, if there is a unimodular constant $c$ such that  $\overline{\phi(z)}-c\phi(z) $ is constant for $z\in \mathbb{D}$, we have
$$H_{\bar{\phi}}=cH_\phi, $$
to get
$$H_{\bar{\phi}}H_{\bar{\phi}}^*=|c|^2H_{ {\phi}} H_{ {\phi}}^*=H_{ {\phi}} H_{ {\phi}}^*.$$
Thus (\ref{eqnf0}) gives that $S_\phi$ is normal.

If there exist $u$ and $f$ in $  H^{\infty}(\mathbb{D})$ with $\|u\|_{\infty}\leq1$ and $|u(0)|<1$ such that
  $$\phi (z)=\frac{\overline {u (z)}f(z)+\overline{f(z)}}{1-|u(z)|^2},$$
  then take the conjugate on both sides of the above equality to get
  $$\overline{\phi (z)}=\frac{  {u (z)}\overline{f(z)}+ {f(z)}}{1-|u(z)|^2}.$$
  Thus we have
  \begin{eqnarray*}
  	\overline{\phi (z)}-u(z) {\phi (z)}&=&\frac{ [u(z)\overline{f(z)}+ {f(z)}] -u(z)[\overline {u (z)}f(z)+\overline{f(z)}]}{1-|u(z)|^2}\\
  	&=&f(z),
  \end{eqnarray*}
  to get
    $$H_{\bar{\phi}}=H_{u   {\phi  }+f}=H_{u   {\phi  } }=H_{\phi}T_{u}.$$
   The second equality follows from that $H_f=0$ as $f$ is in $H^{\infty}$ and the third equality follows from  (\ref{HST})
   as $u$ is in $H^{\infty}(\mathbb{D})$.
   This implies
    $$H_{\bar{\phi}}H_{\bar{\phi}}^*=H_{\phi}T_uT_{u}^*H_{\phi}^*\leq H_{\phi} H_{\phi}^*. $$
    The last inequality follows from
    $$T_uT_u^*\leq T_{u}^*T_u=T_{|u|^2}\leq I.$$
   Thus (\ref{eqnf0}) gives that $S_\phi$ is hyponormal.
This completes the proof.
\end{proof}

\begin{remark}
	\begin{enumerate}[(a)]
		\item If $f(z)\neq 0$ for  some $z\in \mathbb{D}$,  as
		$$\phi (z)=\frac{\overline{u(z)}f(z)+\overline{f(z)}}{1-|u (z)|^2},$$
		Theorem \ref{thm1} tells us that $S_{\phi}$ is hyponormal but not normal.
		\item By virtue of Theorem \ref{thm1}, one can construct numerous hyponormal dual
Toeplitz operators with non-harmonic symbols. For example, take
 $f(z)=z$ and $u(z)=\frac{z}{2},$ let
		$$\phi (z)=\frac{\overline{u(z)}f(z)+\overline{f(z)}}{1-|u (z)|^2}=\frac{\frac{|z|^2}{2}+\bar{z}}{1-\frac{|z|^2}{4}}.$$
		Clearly $\phi$ is not harmonic on the unit disk but Theorem \ref{thm1} implies that $S_{\phi}$ is hyponormal.
	\end{enumerate}
\end{remark}

 Theorem \ref{thm1} immediately gives that hyponormal dual Toeplitz operators with radial symbols  must be normal.

 \begin{proof}[{\bf Proof of Theorem \ref{thm5}}]
  Suppose that $S_{\phi}$ is hyponormal. Then by Theorem \ref{theorem3.2}, there exists a function $u\in H^{\infty}(\mathbb{D})$ with $\|u\|_{\infty}\leq 1$ such that $\bar{\phi}-u\phi$ is analytic. Since $\phi(z)=\phi(|z|)=\phi(r)$ for each $z\in \mathbb{D}$ with polar coordinates representation $z=re^{i\theta}$,  the Cauchy Integral Formula gives
 	\begin{align*}
 		\bar{\phi}(0)-u(0)\phi(0)
 		&=\frac{1}{2\pi}\int_0^{2\pi}\left(\bar{\phi}(re^{i\theta})-u(re^{i\theta})\phi(re^{i\theta})\right)d\theta\\
 		&=\frac{1}{2\pi}\int_0^{2\pi}\left(\bar{\phi}(r)-u(re^{i\theta})\phi(r)\right)d\theta\\
 		&=\bar{\phi}(r)-\frac{1}{2\pi}\int_0^{2\pi}u(re^{i\theta})d\theta\cdot\phi(r)\\
 		&=\bar{\phi}(z)-u(0)\phi(z).
 	\end{align*}
  The first equality holds because $\overline{\phi (z)}-u(z)\phi (z)$ is analytic on the unit disk and the last equality follows from that $u(z)$ is analytic on the unit disk. Corollary 2.7 in \cite{SZ2002} implies that $S_{\phi}$ is normal.
 \end{proof}

 Next, using Theorem \ref{thm1} we will present the proof of Theorem \ref{Cor00}.

\begin{proof}[{\bf Proof of Theorem \ref{Cor00}}]
  Suppose $\phi$ is a bounded and harmonic function on the unit disk. Then $\phi$ can be written as a sum of an analytic function and the conjugate of another analytic function on the unit disk, i.e.,
  $$\phi (z)=f(z)+\overline{g(z)}$$
  for $z\in \mathbb D$ where $f$ and $g$ are analytic on the unit disk. As $f=P(\phi)$ and the Bergman projection $P$ is a bounded operator from $L^{\infty} (\mathbb{D})$ to $BMOA$, we have that both $f$ and $g$ are in $BMOA$.
  Take the conjugate on both sides of the above equality to get
   $$\overline{\phi (z)}=\overline{f(z)}+ {g(z)}$$
   for $z\in \mathbb D$. For a function $u\in H^{\infty}(\mathbb{D})$ with $\|u\|_{\infty}\leq 1$,
   $$\overline{\phi (z)}-u(z)\phi (z)=\overline{f(z)}-u(z)\overline{g(z)}+g(z)-u(z)f(z)$$
   is analytic on the unit disk if and only if
   $\overline{f(z)}-u(z)\overline{g(z)}$ is analytic on the unit disk if and only if
   $$\frac{\partial}{\partial \bar{z}}[\overline{f(z)}-u(z)\overline{g(z)}]=0,$$
   for $z\in \mathbb D$. The above equation is equivalent to the following equation
   \begin{equation}\label{eqnf11}
   	 \overline{f^{\prime}(z)}-u(z)\overline{g^{\prime}(z)}  =0,
   \end{equation}
   for $z\in \mathbb D$.

   If neither $f$ nor $g$ is a constant, the above equation holds if and only if
   $$u(z) =\frac{\overline{f^{\prime}(z)}}{\overline{g^{\prime}(z)}}$$
   for $z\in \mathbb D$ except for countable points.  In this case,  $u$ must be a constant $\overline{\alpha}$ with $|\alpha |\leq 1.$ Thus
   $$f(z)=\alpha g(z)+\beta$$
   for $z\in \mathbb D$ and some constant $\beta .$ So
   $$\phi (z)= \alpha g(z)+\beta +\overline{g(z)}.$$

    If $g$ is a constant, (\ref{eqnf11}) gives that $f$ is constant and hence $\phi$ is constant.
    If $f$ is a constant, then (\ref{eqnf11}) holds if and only if either $u=0$ or $g$ is a constant. Therefore, $\phi (z) =\overline{g(z)}+\beta$ for some constant $\beta$ in the case that $u=0$ or $\phi$ is a constant in the case that $g$ is a constant.

    To summarize the above discussion, we get that
    the set $$\{u\in H^{\infty}(\mathbb{D}):\|u\|_{\infty}\leq1~\text{and}~\bar{\phi}-u\phi\in H^{\infty}(\mathbb{D})\}$$ is not empty if and only if
  there are constants $\alpha$ and $\beta$ with $|\alpha|\leq1$ and a function $g$ in $BMOA$ such that
  $\phi=\alpha g+\bar{g}+\beta$.  By Theorem \ref{theorem3.2},  we have that $S_\phi$ is hyponormal if and only if there are constants $\alpha$ and $\beta$ with $|\alpha|\leq1$ and a function $g$ in $BMOA$ such that
  $\phi=\alpha g+\bar{g}+\beta$.

  Moreover, if $\alpha$ is a unimodular constant $e^{it}$ for some real number $t$,
  let $h=e^{i\frac{t}{2}}g.$ Then $h$ is still in $BMOA$ and
  $$\phi=\alpha g+\bar{g}+\beta =e^{i\frac{t}{2}}(h+\bar{h})+\beta.$$
  Since $\phi$ is bounded on the unit disk, we have that $\Re h$ is in $L^{\infty}(\mathbb{D}).$

  This completes the proof.

\end{proof}

\begin{remark}\label{rmk11}
Theorem \ref{Cor00} demonstrates that dual Toeplitz operators with nonconstant analytic symbols cannot be hyponormal while dual Toeplitz operators with conjugate analytic symbols are hyponormal, indeed they are subnormal as they equal the restriction of the multiplication operator   on $(L^2_a)^{\perp}.$
Comparing Corollary 2.7 in \cite{SZ2002} with Theorem \ref{Cor00},   we see that  some hyponormal dual Toeplitz operators with bounded harmonic symbols may not  be normal.
 Indeed, for $\phi=\alpha g+\bar{g}+\beta,$  $S_{\phi}$ is normal if and only if  $|\alpha|=1$.
\end{remark}

In order to present the proof of Theorem \ref{poly}, we need the following important lemma.

\begin{lemma}\label{Phizw} Suppose   $\Phi (z, \bar{z})$ is a real analytic function on the unit disk, i.e.,
	$ \sum\limits_{i=0}^{\infty}\sum\limits_{j=0}^\infty a_{ij}z^i{\bar{z}}^j$
	converges to $\Phi (z, \bar{z})$.
	If $\Phi (z, \bar{z})\equiv 0$ for $z\in \mathbb{D}$, then
	$$\Phi (z, \bar{w})\equiv 0,$$
	for $z, w\in\mathbb{D}.$
\end{lemma}

\begin{proof}   Suppose   $\Phi (z, \bar{z})$ is a real analytic function on the unit disk and
	$$	\Phi (z, \bar{z}) =\sum\limits_{i=0}^{\infty}\sum\limits_{j=0}^\infty a_{ij}z^i{\bar{z}}^j.$$
	Take partial derivatives $\partial_z^i\partial_{\bar{z}}^j$ on both sides of the above equality and then evaluate value at $z=0$ and $\bar{z}=0$ to get
	$$\partial_z^i\partial_{\bar{z}}^j[\Phi (z, \bar{z})]|_{z=0,\bar{z}=0}=i!
	 j!a_{ij}.$$
	If $\Phi (z, \bar{z})\equiv 0$ for $z\in \mathbb{D}$, then 	
    $$\partial_z^i\partial_{\bar{z}}^j[\Phi (z, \bar{z})]|_{z=0,\bar{z}=0}=0.$$
	Thus $a_{ij}=0$ for $i, j=0, 1, \cdots , $ and so
	$$	\Phi (z, \bar{w}) =\sum\limits_{i=0}^{\infty}\sum\limits_{j=0}^\infty a_{ij}z^i{\bar{w}}^j=0.$$
	This completes the proof.
\end{proof}

  \begin{proof}[{\bf Proof of Theorem \ref{poly}}]  Suppose $\phi$ is a polynomial of $z$ and $\bar{z}$:
  	$$\phi (z)=\sum_{i=0}^{n}\sum_{j=0}^n a_{ij}z^i{\bar{z}}^j.$$
  We write $\phi (z)$ in the following two forms:
  \begin{eqnarray*}
  	\phi (z)&=& g(z)+\overline{h(z)}+\sum_{j=1}^n g_j(z)\bar{z}^j \\
  	&=& g(z)+\overline{h(z)}+\sum_{j=1}^n\overline{h_j(z)}z^j,
  \end{eqnarray*}
  	where
  	\begin{eqnarray*}
  		 g(z)&=&\sum_{i=0}^na_{i0}z^i\\
  		 h(z)&=&\sum_{j=1}^n\overline{a_{0j}}{z}^j\\
  		 g_j(z)&=&\sum_{i=1}^na_{ij}z^i\\
  		 h_j(z) &=&\sum_{i=1}^n\overline{a_{ji}}z^i
  	\end{eqnarray*}  for $j=1, \cdots , n.$
  	
  First	suppose  $S_{\phi}$ is hyponormal.  By Theorem \ref{thm1}, we have one of the following two conditions holds:
  \begin{enumerate}[(a)]
  	\item  there exist  a function $f\in  H^\infty (\mathbb{D})$ and a function $u\in H^\infty (\mathbb{D})$ with     $\|u\|_{\infty}\leq 1$ and $|u(0)|<1$ such that  $$\phi (z)=\frac{\overline{u(z)}f(z)+\overline{f(z)}}{1-|u (z)|^2}$$
  	for $z\in \mathbb{D};$
  	\item  there is a unimodular constant $c$ such that  $\overline{\phi(z)}-c\phi(z) $ is a constant for $z\in \mathbb{D}$.
  \end{enumerate}
  If Condition (b) holds,
   there is a unimodular constant $c$ such that
   \begin{eqnarray*}
   	  \sum_{i=0}^{n}\sum_{j=0}^n \overline{a_{ij}}\bar{z}^i{ {z}}^j-c \left(\sum_{i=0}^{n}\sum_{j=0}^n a_{ij}z^i{\bar{z}}^j \right)
   	 =\sum_{i=0}^{n}\sum_{j=0}^n(\overline{ a_{ji}}-ca_{ij})z^i{\bar{z}}^j
   \end{eqnarray*}
   is a constant. This implies that the dual Toeplitz operator $S_{\phi}$ is normal. As $\{z^i\bar{z}^j:i, j=0, \cdots , n\}$ are linear independent, we have
   $$\overline{ a_{ji}}=ca_{ij}$$
   for     $(i, j)\in \{0, \cdots, n\}\times \{0, \cdots, n\}$ with $(i, j)\neq (0, 0)$.

  If Condition (a) holds,
   in the proof of Theorem \ref{thm1},	we have shown that
  	  \begin{eqnarray*}
  	  f(z)&=&\overline{\phi (z)}-u(z)\phi (z)\\
  	      &=& \overline{g(z)}+ {h(z)}+\sum_{j=1}^n {h_j(z)}{\bar{z}}^j-u(z)[g(z)+\overline{h(z)}+\sum_{j=1}^n g_j(z)\bar{z}^j]
  	  \end{eqnarray*}
   Let
   $$\Phi (z, \bar{z}) =f(z)-\left[\overline{g(z)}+ {h(z)}+\sum_{j=1}^n {h_j(z)}{\bar{z}}^j-u(z)\big[g(z)+\overline{h(z)}+\sum_{j=1}^n g_j(z)\bar{z}^j\big]\right]. $$
   As $\Phi (z, \bar{z})\equiv 0$ for $z\in \mathbb{D},$ Lemma \ref{Phizw} implies that $\Phi (z, \bar{w})\equiv 0$ for $z, w\in \mathbb{D},$ and hence
  \begin{equation}\label{Phizw2}
 f(z)= \overline{g(w)}+ {h(z)}+\sum_{j=1}^n {h_j(z)}{\bar{w}}^j-u(z)\big[g(z)+\overline{h(w)}+\sum_{j=1}^n g_j(z)\bar{w}^j\big],
  \end{equation}  for $z, w\in\mathbb{D}.$

       Take partial derivatives  $\partial_{\bar{w}}^m$  on both sides of Equation (\ref{Phizw2}) for $1 \leq m\leq n$ to get
    \begin{eqnarray*}
    	0 &=& \overline{g^{(m)}(w)}+\sum_{j=m}^n  {h_j(z)}\frac{j!}{(j-m)!}\bar{w}^{j-m}-u(z)\big[\overline{ h^{(m)}(w)}+\sum_{j=m}^n  {g_j(z)}\frac{j!}{(j-m)!}\bar{w}^{j-m}\big]\\
    	&=&  \sum_{j=m}^n \big[\frac{\overline{g^{(j)}(0)}}{(j-m)!}+ {h_j(z)}\frac{j!}{(j-m)!}\big]\bar{w}^{j-m}-u(z)\left[ \sum_{j=m}^n \big[\frac{\overline{ h^{(j)}(0)}}{(j-m)!}+ {g_j(z)}\frac{j!}{(j-m)!}\big]\bar{w}^{j-m}\right]
    \end{eqnarray*}
Evaluating both sides of the above equation at $w=0$ we have
 \begin{eqnarray*}
 	0 &=& \overline{g^{(m)}(0)}+  {h_m(z)}m!-u(z)[\overline{ h^{(m)}(0)}+  {g_m(z)}m!]\\
 	  &=& \overline{a_{m0}}m!+{h_m(z)}m!-u(z)[ a_{0m}m!+g_{m}(z)m!]
 \end{eqnarray*}
Thus for $z\in \mathbb{D}$, we have
\begin{eqnarray*}
 u(z)\big({a_{0m} +g_{m}(z) }\big)&=& {\overline{a_{m0}} +{h_m(z)} },
\end{eqnarray*}
equivalently,
\begin{eqnarray*}
 u(z)\big({a_{0m} + \sum_{i=1}^na_{im}z^i }\big)&=& {\overline{a_{m0}} +\sum_{i=1}^n\overline{a_{mi}}z^i },
\end{eqnarray*}
thus
\begin{eqnarray*}
u(z)\left({  \sum_{i=0}^na_{im}z^i }\right) &=&  { \sum_{i=0}^n\overline{a_{mi}}z^i }
\end{eqnarray*}
for each $m\in\{1, \cdots , n\}.$  The last equation is (\ref{uexp}) as desired.

    Conversely, if Condition (i) in Theorem \ref{poly} holds, then $S_{\phi}$ is normal and hence hyponormal.

    Suppose (\ref{uexp}) holds for a function $u\in H^{\infty}(\mathbb{D})$ with $\|u\|_{\infty}\leq 1$ and $|u(0)|<1$. For each $m\in\{1, \cdots , n\}$, we have
   \begin{equation}\label{ugh}
   	 u(z)\left(\sum_{i=0}^na_{im}z^i\right)=\sum_{i=0}^n\overline{a_{mi}}z^i,
   \end{equation}
    to get
    \begin{eqnarray*}
    	u(z)\phi (z)&=& u(z)\left(\sum_{i=0}^{n}\sum_{j=0}^n a_{ij}z^i{\bar{z}}^j\right)\\
    	&=&\sum_{j=0}^{n}u(z)\left(\sum_{i=0}^n a_{ij}z^i\right){\bar{z}}^j\\
    		&=&u(z)\left(\sum_{i=0}^n a_{i0}z^i\right)+\sum_{j=1}^{n}u(z)\left(\sum_{i=0}^n a_{ij}z^i\right){\bar{z}}^j\\
    			&=&u(z)\left(\sum_{i=0}^n a_{i0}z^i\right)+\sum_{j=1}^{n}\sum_{i=0}^n\overline{a_{ji}}z^i{\bar{z}}^j\\
    			&=&u(z)\left(\sum_{i=0}^n a_{i0}z^i\right)-\sum_{i=0}^n\overline{a_{0i}}z^i+\sum_{j=0}^{n}\sum_{i=0}^n\overline{a_{ji}}z^i{\bar{z}}^j\\
    			&=&u(z)\left(\sum_{i=0}^n a_{i0}z^i\right)-\sum_{i=0}^n\overline{a_{0i}}z^i+\overline{\phi (z)}.
    		\end{eqnarray*}
    	The fourth equality follows from (\ref{ugh}) and the sixth equality follows from
    	$$\overline{\phi(z)}=  \sum_{j=0}^{n}\sum_{i=0}^n\overline{a_{ji}}z^i{\bar{z}}^j.$$   Thus
    	$$\overline{\phi (z)} -u(z)\phi (z)=-u(z)\left(\sum_{i=0}^n a_{i0}z^i\right)+\sum_{i=0}^n\overline{a_{0i}}z^i$$
    	is analytic on the unit disk $\mathbb{D}$.  By Theorem \ref{theorem3.2}, the dual Toeplitz operator $S_{\phi}$ is hyponormal.


\end{proof}

\begin{remark}
	We will show two examples that hyponormal dual Toeplitz operators with
	polynomial of $z$ and $\bar{z}$ are not normal.
	\begin{enumerate}[(a)]
	  \item If $g(z)$ is a polynomial of $z$ and $\alpha$ is a constant with $0<|\alpha |<1$, let $\phi(z)=\alpha g(z)+\overline{g(z)}$. Then $\phi(z)$ is a harmonic polynomial of $z$ and $\bar{z}$.  Theorem \ref{Cor00} gives that $S_{\phi}$ is hyponormal but not normal.
	
      \item The following example provides  a hyponormal dual Toeplitz operator with non-harmonic polynomial symbol of $z$ and $\bar{z}$ that is not normal.

Let $\phi(z)=1 +5z+10\bar{z} +|z|^2$. Clearly, $\phi$ is a non-harmonic polynomial of $z$ and $\bar z$. Then $\phi$ satisfies Condition (ii) in the above theorem. In fact, let
$$u(z)=\frac{5+z}{10+z}.$$
Thus we have
\begin{align*}
	\overline{\phi(z)}-u(z)\phi(z)
    &= 1+5\bar{z}+10z+|z|^2- \frac{5+z}{10+z}(1 +5z+10 \bar{z} +|z|^2 )\\
	&=1+10z+(5+z)\bar{z}-\frac{(5+z)(1+5z)}{10+z}-\frac{(5+z)(10+z )\bar{z}}{10+z}\\
	&=1+10z-\frac{(5+z)(1+5z)}{10+z}
\end{align*}
is bounded and analytic on the unit disk and
\begin{eqnarray*}
	|u(z)|&=&\left|\frac{5+z}{10+z}\right|\\
	&\leq & \frac{1}{2}+\left |\frac{z}{2(10+z)}\right|\\
	&\leq & \frac{1}{2}+\frac{1}{2(10-1)} <\frac{3}{4}
\end{eqnarray*}
for $z\in \mathbb{D}.$ So  we have that $\|u\|_{\infty}<1$ and $|u(0)|=\frac{1}{2}, $ to get that $S_{\phi}$ is hyponormal but not normal.
	
\end{enumerate}
\end{remark}

\section{subnormal and $\varTheta$-class dual Toeplitz operators}\label{sec3}

Regarding Question \ref{Q2}, Theorem \ref{DSNC} provides an affirmative answer for dual Toeplitz operators with harmonic polynomial symbols.

  In \cite{Ito1972}, Ito and Wong presented the interesting
example   $T_{z+\frac{1}{2}\bar{z}}$ on the Hardy space $H^2$. This is an easy example of a hyponormal operator
that is not subnormal.
For dual Toeplitz operators, by showing that some of these operators with harmonic polynomial symbols are
hyponormal but not subnormal, the following result reveals a clear distinction between the
subnormal and hyponormal operator classes.

\begin{proof}[\bf Proof of Theorem \ref{DSNC}.]
	Suppose that $\phi$ is a   harmonic polynomial on the unit disk.
Since any normal operator is  subnormal, and $M_{\phi}$ provides a natural normal extension for the dual Toeplitz operator $S_{\phi}$ whenever $\phi$ is co-analytic, the sufficiency follows immediately.

For the necessity, suppose $S_{\phi}$ is subnormal.
Then $S_{\phi}$ must be hyponormal.
By Theorem \ref{Cor00}, there exist an analytic polynomial $g$ and constants $\alpha, \beta$ with $|\alpha|\leq 1$ such that
$\phi=\alpha g + \bar{g} + \beta$.
If $|\alpha|= 1$, then $S_{\phi}$ is normal by Remark \ref{rmk11}, thus is subnormal.
If $\alpha= 0$, then $\phi=\bar{g} + \beta$ is co-analytic,
in this case, $M_{\phi}$ serves as a normal extension for $S_{\phi}$, which implies that $S_{\phi}$ is subnormal.

To complete the proof of Theorem \ref{DSNC},  we will show that $S_{\phi}$ is not subnormal for $\alpha \in \mathbb{C}$ with $0<|\alpha|<1$.
Recalling that the $k$-th power of any subnormal operator is hyponormal for all $k \in \mathbb{Z}^+$,
it suffices to show that $S_{\phi}^2$ is not hyponormal.
Furthermore, since hyponormality and subnormality are invariant under translation by any scalar,
without loss of generality, we may assume that $\beta =0$ and $g(0)=0$.
Suppose that $g$ is an analytic polynomial of degree $n$, given by
$g(z)=\sum \limits_{m=1}^{n} a_{m} z^{m} $ with $a_{n} \neq 0$.
For convenience, we may assume that $\alpha \in (0, 1)$.
Indeed, if $\alpha$ is complex, we can choose a unimodular constant $\xi$ (i.e., $|\xi|=1$) such that $\alpha \xi^2$ is real. Considering the symbol $\tilde{\phi} =  \alpha (\xi g)+\overline{\xi g}=\bar{\xi}(\alpha \xi^2g+\bar{g})$ then reduces the problem to the real case, as the hyponormality of $S_{\phi}^2$ is invariant under such a rotation.

By the linearity of the dual Toeplitz operator with respect to its symbol and the identity (\ref{SH}), we have
\begin{align}\label{eqp0}
   S_{\phi}^2
   =(\alpha S_{g}+S_{\bar{g}})^2
   =\alpha^2 S_{g^2}+\alpha S_{g\bar{g}}+\alpha S_{\bar{g}}S_{g}+S_{\bar{g}^2}.
 \end{align}
 Similarly, by $S_{\phi}^*=S_{\bar{\phi}}$ and the identity (\ref{SH}), we obtain
 \begin{align}\label{eqp1}
   (S_{\phi}^2)^{*}
   =S_{\bar{\phi}}^2
   =(\bar{\alpha}S_{\bar{g}}+S_{g})^2
   =\alpha^2S_{\bar{g}^2}+\alpha S_{\bar{g}}S_{g}+\alpha S_{g\bar{g}}+S_{g^2}.
 \end{align}
To show that $S_{\phi}^2$ is not hyponormal, we seek a test function $f(z) = \bar{z}^n + t\bar{z}^{3n}$ for some $t \in \mathbb{C}$ such that $\|S_{\phi}^2 f\| < \|(S_{\phi}^2)^* f\|$.
Since
 \begin{align}\label{eqp2}
   P(z^m\bar{z}^k)
   =\left\{
    \begin{array}{lc}
        \frac{m-k+1}{m+1}z^{m-k}, & m \geq k \\
        0,  & m<k\\
    \end{array},
      \right.
 \end{align}
a direct calculation gives that
 \begin{align*}
   A:=S_{g^2}\bar{z}^n
   =(I-P)(g^2 \bar{z}^n)
   =g^2 \bar{z}^n-P(g^2 \bar{z}^n);
 \end{align*}
 \begin{align*}
   B:=S_{g\bar{g}}\bar{z}^n
   =(I-P)(g\bar{g}\bar{z}^n)
   =g\bar{g}\bar{z}^n,
 \end{align*}
 as $P(g \bar{g} \bar{z}^{n})=0$;

 \begin{align*}
   S_{\bar{g}}S_{g}\bar{z}^n
   &=S_{\bar{g}}(I-P)(g\bar{z}^n) \nonumber \\
   &=S_{\bar{g}}(g\bar{z}^n-a_n \cdot \frac{1}{n+1}) \nonumber \\
   &=(I-P)(\bar{g}g\bar{z}^n-\frac{a_n}{n+1}\bar{g}) \nonumber \\
   &=g\bar{g}\bar{z}^n-\frac{a_n}{n+1}\bar{g}\nonumber \\
   &=B-\frac{a_n}{n+1}\bar{g};
 \end{align*}

 \begin{align*}
   C&:=S_{\bar{g}^2}\bar{z}^n =(I-P)(\bar{g}^{2}\bar{z}^n)=\bar{g}^{2}\bar{z}^n;
 \end{align*}
 \begin{align*}
   D&:=S_{g^2}\bar{z}^{3n} =(I-P)(g^2 \bar{z}^{3n}) =g^2 \bar{z}^{3n};
 \end{align*}
 \begin{align*}
   E&:=S_{g\bar{g}}\bar{z}^{3n}=(I-P)(g\bar{g}\bar{z}^{3n}) =g\bar{g}\bar{z}^{3n};
 \end{align*}
 \begin{align*}
   S_{\bar{g}}S_{g}\bar{z}^{3n} &=S_{\bar{g}}(I-P)(g\bar{z}^{3n})=S_{\bar{g}}(g\bar{z}^{3n})=(I-P)(\bar{g}g\bar{z}^{3n})
   =g\bar{g}\bar{z}^{3n}=E;
 \end{align*}
 \begin{align*}
   F&:=S_{\bar{g}^2}\bar{z}^{3n}=(I-P)(\bar{g}^{2}\bar{z}^{3n}) =\bar{g}^{2}\bar{z}^{3n}.
 \end{align*}
  Therefore, we have
  \begin{align*}
    S_{\phi}^2f&=(\alpha^2 S_{g^2}+\alpha S_{g\bar{g}}+\alpha S_{\bar{g}}S_{g}+S_{\bar{g}^2})(\bar{z}^n)\\
    &+(\alpha^2 S_{g^2} +\alpha S_{g\bar{g}}+\alpha S_{\bar{g}}S_{g}+S_{\bar{g}^2})(t\bar{z}^{3n})\\
    &=\alpha^2A+\alpha B+\alpha(B-\frac{a_n}{n+1}\bar{g})+C+\alpha^2 tD+2\alpha tE+tF\\
    &=\alpha^2A+2\alpha B-\frac{\alpha a_n}{n+1}\bar{g}+C+\alpha^2 tD+2\alpha tE+tF,
  \end{align*}
   \begin{align*}
    (S_{\phi}^2)^{*}f&=(\alpha^2S_{\bar{g}^2}+\alpha S_{\bar{g}}S_{g}+\alpha S_{g\bar{g}}+S_{g^2})(\bar{z}^n)\\
    &+(\alpha^2S_{\bar{g}^2}+\alpha S_{\bar{g}}S_{g}+\alpha S_{g\bar{g}}+S_{g^2})(t\bar{z}^{3n})\\
    &=\alpha^2C+\alpha(B-\frac{a_n}{n+1}\bar{g})+\alpha B+A+\alpha^2 tF+2\alpha t E+tD\\
    &=\alpha^2C+2\alpha B-\frac{\alpha a_n}{n+1}\bar{g}+A+\alpha^2 tF+2\alpha t E+tD\\
    &=A+2\alpha B-\frac{\alpha a_n}{n+1}\bar{g}+\alpha^2C+tD+2\alpha t E+\alpha^2 tF.
  \end{align*}
Expanding the expression via the sesquilinearity of the inner product and collecting like terms, we obtain
\begin{align*}
 & \| S_{\phi}^2f\|^2 -\| (S_{\phi}^2)^{*}f\|^2 \\
 &=\langle S_{\phi}^2f,S_{\phi}^2f \rangle -\langle (S_{\phi}^2)^{*}f,(S_{\phi}^2)^{*}f \rangle\\
 &=(\alpha^4-1)\|A\|^2+2\alpha(\alpha^2-1)\langle A, B\rangle +\frac{\alpha \overline{a_n}}{n+1}(1-\alpha^2)\langle A, \bar{g}\rangle \\
 &+(\alpha^4-1)\bar{t}\langle A, D\rangle+2\alpha\bar{t}(\alpha^2-1)\langle A, E\rangle
 +2\alpha(\alpha^2-1)\langle B, A\rangle\\
 &+2\alpha(1-\alpha^2)\langle B, C\rangle+2\alpha\bar{t}(\alpha^2-1)\langle B, D\rangle+2\alpha\bar{t}(1-\alpha^2)\langle B, F\rangle\\
 &+\frac{\alpha a_n}{n+1}(1-\alpha^2)\langle \bar{g}, A\rangle+\frac{\alpha a_n}{n+1}(\alpha^2-1)\langle \bar{g}, C\rangle
 +\frac{\alpha\bar{t} a_n}{n+1}(1-\alpha^2)\langle \bar{g}, D\rangle\\
 &+\frac{\alpha\bar{t} a_n}{n+1}(\alpha^2-1)\langle \bar{g}, F\rangle +2\alpha(1-\alpha^2)\langle C, B \rangle +\frac{\alpha \overline{a_n}}{n+1}(\alpha^2-1)\langle C, \bar{g}\rangle \\
 &+(1-\alpha^4)\|C\|^2 +2\alpha\bar{t}(1-\alpha^2)\langle C, E\rangle+(1-\alpha^4)\bar{t}\langle C, F\rangle\\
 &+(\alpha^4-1)t\langle D, A\rangle+2\alpha t(\alpha^2-1)\langle D, B\rangle
 +\frac{\alpha t \overline{a_n}}{n+1}(1-\alpha^2)\langle D, \bar{g}\rangle \\
 &+|t|^2(\alpha^4-1)\|D\|^2 +2\alpha|t|^2(\alpha^2-1)\langle D, E\rangle+2\alpha t(\alpha^2-1)\langle E, A\rangle\\
 &+2\alpha t(1-\alpha^2)\langle E, C\rangle+2\alpha |t|^2(\alpha^2-1)\langle E, D\rangle+ 2\alpha |t|^2(1-\alpha^2)\langle E, F\rangle\\
 &+ 2\alpha t(1-\alpha^2)\langle F, B\rangle+\frac{\alpha t\overline{a_n}}{n+1}(\alpha^2-1)\langle F, \bar{g}\rangle
 + t(1-\alpha^4)\langle F, C\rangle\\
 &+2\alpha |t|^2(1-\alpha^2)\langle F, E\rangle +|t|^2(1-\alpha^4)\|F\|^2.
\end{align*}
By comparing terms with opposite coefficients and direct computations, we have
\begin{align*}
  \|A\|^2=\|C\|^2-\|P(g^2\bar{z}^n)\|^2; \quad  \|D\|^2=\|F\|^2;
\end{align*}
\begin{align*}
  \langle A, B\rangle=\langle B, C\rangle ; \quad \langle B, A\rangle=\langle C, B\rangle;
\end{align*}
\begin{align*}
  \langle A, D\rangle=\langle C, F\rangle ; \quad \langle D, A\rangle=\langle F, C\rangle;
\end{align*}
\begin{align*}
  \langle A, E\rangle=\langle B, F\rangle ; \quad \langle E, A\rangle=\langle F, B\rangle;
\end{align*}
\begin{align*}
  \langle B, D\rangle=\langle C, E\rangle ; \quad \langle D, B\rangle=\langle E, C\rangle;
\end{align*}
\begin{align*}
  \langle D, E\rangle=\langle E, F\rangle ; \quad \langle E, D\rangle=\langle F, E\rangle;
\end{align*}
\begin{align*}
  \langle \bar{g}, A\rangle=\langle \bar{g}, g^2\bar{z}^n-P(g^2\bar{z}^n)\rangle
  =\langle \bar{g}, g^2\bar{z}^n\rangle
  =\langle \bar{g}^3, \bar{z}^n\rangle
  =\left(\sum_{\substack{i+j+m=n \\ 1 \leq i, j, m \leq n}} \overline{a}_{i} \overline{a}_{j} \overline{a}_{m}\right)
\frac{1}{n+1};
\end{align*}
\begin{align*}
  \langle \bar{g}, C\rangle=\langle \bar{g}, \bar{g}^2\bar{z}^n\rangle =0;
\end{align*}
\begin{align*}
  \langle \bar{g}, D\rangle
  =\langle \bar{g}, g^2\bar{z}^{3n}\rangle
  =\langle \bar{g}^3, \bar{z}^{3n}\rangle
  =\overline{a_n}^3\frac{1}{3n+1};
\end{align*}
\begin{align*}
  \langle \bar{g}, F\rangle=\langle \bar{g}, \bar{g}^2\bar{z}^{3n}\rangle =0.
\end{align*}
It follows that
\begin{align*}
   \| S_{\phi}^2f\|^2 -\| (S_{\phi}^2)^{*}f\|^2
  &=(1-\alpha^4)(\|C\|^2-\|A\|^2)\\
  &+\left[ \frac{\alpha \overline{a_n}}{n+1}(1-\alpha^2)\langle A, \bar{g}\rangle+\frac{\alpha a_n}{n+1}(1-\alpha^2)\langle \bar{g}, A\rangle \right]\\
  &+\left[ \frac{\alpha t \overline{a_n}}{n+1}(1-\alpha^2)\langle D, \bar{g}\rangle+\frac{\alpha\bar{t} a_n}{n+1}(1-\alpha^2)\langle \bar{g}, D\rangle \right]\\
  &=(1-\alpha^4)\|P(g^2\bar{z}^n)\|^2+\frac{\alpha(1-\alpha^2)}{n+1}\cdot2Re(a_n\langle \bar{g}, A\rangle)\\
  &+\frac{\alpha(1-\alpha^2)}{n+1}\cdot2Re(t\overline{a_n}\langle D, \bar{g}\rangle)\\
  &=(1-\alpha^4)\|P(g^2\bar{z}^n)\|^2+\frac{2\alpha(1-\alpha^2)}{(n+1)^2}Re\left[a_n\left(\sum_{\substack{i+j+m
  =n \\ 1 \leq i, j, m \leq n}} \overline{a}_{i} \overline{a}_{j} \overline{a}_{m}\right)\right]\\
  &+\frac{2|a_n|^2\alpha(1-\alpha^2)}{(n+1)(3n+1)}Re(ta_n^2).
\end{align*}
For $0<\alpha<1$, then $\| S_{\phi}^2f\|^2 -\| (S_{\phi}^2)^{*}f\|^2 <0 $ holds if and only if
\begin{align*}\label{Inq333}
  & (1+\alpha^2)\|P(g^2\bar{z}^n)\|^2+\frac{2\alpha}{(n+1)^2}Re\left[a_n\left(\sum_{\substack{i+j+m
  =n \\ 1 \leq i, j, m \leq n}} \overline{a}_{i} \overline{a}_{j} \overline{a}_{m}\right)\right] \\
  & +\frac{2|a_n|^2\alpha}{(n+1)(3n+1)}Re(ta_n^2)<0.
\end{align*}
Since the above inequality is always solvable for $t$,
there exists $t \in \mathbb{C}$ such that $\| S_{\phi}^2f\|^2 <\| (S_{\phi}^2)^{*}f\|^2$.
The existence of such a $t$ implies that $S_{\phi}^2$ is not hyponormal,
hence $S_{\phi}$ cannot be subnormal.
This completes the proof.
\end{proof}


Recall that the class $\varTheta$ consists of all bounded linear operators $T$ acting on a separable Hilbert space $\mathcal{H}$ such that $T^{\ast}T$ commutes with $T+T^{\ast}$.
In \cite{S1985}, Sun constructed a counter-example for the problem raised by Halmos once he completely characterized $\varTheta$-class Toeplitz operators on the Hardy space.
 A Campbell's theorem was used there:
\begin{theorem}(\cite{Campbell1978})
  If $T\in\varTheta$ and $T$ is hyponormal, then $T$ is subnormal.
\end{theorem}

Motivated by Sun's work \cite{S1985}, and in light of Question \ref{Q1} formulated above, we extend this investigation to the context of $\varTheta$-class dual Toeplitz operators.

For $x, y$ in the same Hilbert space $\mathcal{H}$, let $x \otimes y$ be the rank one operator defined by $(x \otimes y)(z)=\langle z, y \rangle x$ for $z \in \mathcal{H}$. It is easy to verified that $\|x \otimes y\|=\|x\|\|y\|$. If $T$ and $S$ are bounded linear operators on $\mathcal{H}$, then $T(x \otimes y)S^{*}=(Tx) \otimes (Sy)$.
The rank one operator $x \otimes y$ is self-adjoint if and only if there exists $a \in \mathbb{R}$ such that $y=ax$.

For $\varTheta$-class dual Toeplitz operators with bounded harmonic symbols, we obtained the following characterization.
Surprisingly, the $\varTheta$-class membership of dual Toeplitz operators differs significantly from that of Toeplitz operators on the Hardy space.
\begin{theorem}
Let $\phi=f+\bar{g}$ be a bounded harmonic function on $\mathbb{D}$ with $f$ and $g$ analytic. Then the following conditions are equivalent:\\
(a) $S_{\phi} \in \varTheta$;\\
(b) There exist $\alpha, \beta \in \mathbb{C}$ with $|\alpha|=1$ such that $\phi= \alpha g+\bar{g}+\beta$;\\
(c) $S_{\phi}$ is normal.
\end{theorem}
\begin{proof}
It follows from Corollary 2.7 in \cite{SZ2002} or Remark \ref{rmk11} that $(b)\Rightarrow (c)$. Since $(c) \Rightarrow (a)$ is obvious,
we only need to prove $(a)\Rightarrow (b)$.
Suppose that $S_{\phi} \in \varTheta$, then
$$S^{*}_{\phi}S_{\phi}(S^{*}_{\phi}+S_{\phi})-(S^{*}_{\phi}+S_{\phi})S^{*}_{\phi}S_{\phi}=0.$$
A direct calculation gives that
\begin{align*}
  &\quad(S^{*}_{\phi}S_{\phi})(S^{*}_{\phi}+S_{\phi})=S^{*}_{\phi}S_{\phi}S^{*}_{\phi}+ S^{*}_{\phi}S_{\phi}S_{\phi}\\
  &=(S_{\bar{f}}+S_g)(S_{f}+S_{\bar{g}})(S_{\bar{f}}+S_g)+(S_{\bar{f}}+S_g)(S_{f}+S_{\bar{g}})(S_{f}+S_{\bar{g}})\\
  &=S_{\bar{f}}S_{|f|^2}+S_{\bar{f}}S_{fg}+S_{\bar{f}^2\bar{g}}+S_{\bar{f}\bar{g}}S_g
  +S_{|f|^2g}+S_{fg^2}+S_{\bar{f}|g|^2}+S_{|g|^2}S_g\\
  &+S_{\bar{f}}S_{f^2}+S_{\bar{f}}S_{f\bar{g}}+S_{\bar{f}\bar{g}}S_f+S_{\bar{f}\bar{g}^2}
  +S_{f^2g}+S_{f|g|^2}+S_{|g|^2}S_f+S_{\bar{g}|g|^2}
\end{align*}
and
\begin{align*}
  &\quad(S^{*}_{\phi}+S_{\phi})(S^{*}_{\phi}S_{\phi})=[(S^{*}_{\phi}S_{\phi})(S^{*}_{\phi}+S_{\phi})]^{*}\\
  &=S_{|f|^2}S_f+S_{\bar{f}\bar{g}}S_f+S_{f^2g}+S_{\bar{g}}S_{fg}
  +S_{|f|^2\bar{g}}+S_{\bar{f}\bar{g}^2}+S_{f|g|^2}+S_{\bar{g}}S_{|g|^2}\\
  &+S_{\bar{f}^2}S_f+S_{\bar{f}g}S_f+S_{\bar{f}}S_{fg}+S_{fg^2}
  +S_{\bar{f}^2\bar{g}}+S_{\bar{f}|g|^2}+S_{\bar{f}}S_{|g|^2}+S_{g|g|^2}.
\end{align*}
It follows that
\begin{align*}
 &\quad[S^{*}_{\phi}S_{\phi}, S^{*}_{\phi}+S_{\phi}]\\
 &=S_{\bar{f}}S_{|f|^2}+S_{\bar{f}\bar{g}}S_g+S_{|f|^2g}+S_{|g|^2}S_g
 +S_{\bar{f}}S_{f^2}+S_{\bar{f}}S_{f\bar{g}}+S_{|g|^2}S_f+S_{\bar{g}|g|^2}\\
 &-(S_{|f|^2}S_f+S_{\bar{g}}S_{fg}+S_{|f|^2\bar{g}}+S_{\bar{g}}S_{|g|^2}
 +S_{\bar{f}^2}S_f+S_{\bar{f}g}S_f+S_{\bar{f}}S_{|g|^2}+S_{g|g|^2})\\
 &=(S_{\bar{f}}S_{|f|^2}-S_{\bar{f}^2}S_f)+(S_{\bar{f}\bar{g}}S_g-S_{\bar{f}}S_{|g|^2})\\
 &+(S_{|f|^2g}-S_{\bar{f}g}S_f)+(S_{|g|^2}S_g-S_{g|g|^2})\\
 &+(S_{\bar{f}}S_{f^2}-S_{|f|^2}S_f)+(S_{\bar{f}}S_{f\bar{g}}-S_{|f|^2\bar{g}})\\
 &+(S_{|g|^2}S_f-S_{\bar{g}}S_{fg})+(S_{\bar{g}|g|^2}-S_{\bar{g}}S_{|g|^2}).
\end{align*}
Since
\begin{align*}
  S_{\bar{f}|f|^2}&=S_{\bar{f}}S_{|f|^2}+H_{\bar{f}}H_{|f|^2}^{*}=S_{\bar{f}^2}S_f + H_{\bar{f}^2}H_{\bar{f}}^{*};\\
  S_{\bar{f}|g|^2}&=S_{\bar{f}\bar{g}}S_g+H_{\bar{f}\bar{g}}H_{\bar{g}}^{*}=S_{\bar{f}}S_{|g|^2}+H_{\bar{f}}H_{|g|^2}^{*};\\
  S_{g|f|^2}&=S_{\bar{f}g}S_f+H_{\bar{f}g}H_{\bar{f}}^{*};\\
  S_{g|g|^2}&=S_{|g|^2}S_g+H_{|g|^2}H_{\bar{g}}^{*};\\
  S_{f|f|^2}&=S_{\bar{f}}S_{f^2}+H_{\bar{f}}H_{\bar{f}^2}^{*}=S_{|f|^2}S_f+H_{|f|^2}H_{\bar{f}}^{*};\\
  S_{\bar{g}|f|^2}&=S_{\bar{f}}S_{f\bar{g}}+H_{\bar{f}}H_{\bar{f}g}^{*};\\
  S_{f|g|^2}&=S_{|g|^2}S_f+H_{|g|^2}H_{\bar{f}}^{*}=S_{\bar{g}}S_{fg}+H_{\bar{g}}H_{\bar{f}\bar{g}}^{*};\\
  S_{\bar{g}|g|^2}&=S_{\bar{g}}S_{|g|^2}+H_{\bar{g}}H_{|g|^2}^{*},
\end{align*}
we obtain
\begin{equation}\label{H1}
  \begin{aligned}
  0&=[S^{*}_{\phi}S_{\phi}, S^{*}_{\phi}+S_{\phi}]\\
  &=(H_{\bar{f}^2}H_{\bar{f}}^{*}-H_{\bar{f}}H_{|f|^2}^{*})+(H_{\bar{f}}H_{|g|^2}^{*}-H_{\bar{f}\bar{g}}H_{\bar{g}}^{*})\\
  &+H_{\bar{f}g}H_{\bar{f}}^{*}-H_{|g|^2}H_{\bar{g}}^{*}
  +(H_{|f|^2}H_{\bar{f}}^{*}-H_{\bar{f}}H_{\bar{f}^2}^{*})\\
  &-H_{\bar{f}}H_{\bar{f}g}^{*}
  +(H_{\bar{g}}H_{\bar{f}\bar{g}}^{*}-H_{|g|^2}H_{\bar{f}}^{*})+H_{\bar{g}}H_{|g|^2}^{*}.
 \end{aligned}
\end{equation}

Let $\mathcal{S}_w(T)=T-S_{\varphi_{w}}TS_{\bar{\varphi}_{w}}$ for any bounded linear operator $T$ on the orthogonal complement of the Bergman space, where $\varphi_{w}(z)=\frac{w-z}{1-\bar{w}z}$ is the M\"obius transform on $\mathbb{D}.$
Then $\mathcal{S}_w^2(T)=T-2S_{\varphi_{w}}TS_{\bar{\varphi}_{w}}+S_{\varphi_{w}}^2TS_{\bar{\varphi}_{w}}^2$
and $\mathcal{S}_w^2(H_{\varphi}H_{\psi}^{*})=H_{\varphi}(k_w\otimes k_w)H_{\psi}^{*}=(H_{\varphi}k_w)\otimes (H_{\psi}k_w)$ by \cite{SZ1999}, where $k_w(z)=\frac{1-|w|^2}{(1-\bar{w}z)^2}$ is the normalized reproducing kernel of the Bergman space. Let $w=0$, then $k_0(z)=1$.
Apply $\mathcal{S}_0^2$ to both sides of the identity (\ref{H1}), then we obtain
\begin{align*}
  0&=[(H_{\bar{f}^2}1)\otimes (H_{\bar{f}}1)-(H_{\bar{f}}1)\otimes (H_{|f|^2}1)]+[(H_{\bar{f}}1)\otimes (H_{|g|^2}1)-(H_{\bar{f}\bar{g}}1)\otimes (H_{\bar{g}}1)]\\
  &+(H_{\bar{f}g}1)\otimes (H_{\bar{f}}1)-(H_{|g|^2}1)\otimes (H_{\bar{g}}1)
  +[(H_{|f|^2}1)\otimes (H_{\bar{f}}1)-(H_{\bar{f}}1)\otimes (H_{\bar{f}^2}1)]\\
  &-(H_{\bar{f}}1)\otimes (H_{\bar{f}g}1)
  +[(H_{\bar{g}}1)\otimes (H_{\bar{f}\bar{g}}1)-(H_{|g|^2}1)\otimes (H_{\bar{f}}1)]+(H_{\bar{g}}1)\otimes (H_{|g|^2}1).
\end{align*}
Therefore
\begin{align*}
  0&=(H_{\bar{f}^2}1+H_{\bar{f}g}1+H_{|f|^2}1-H_{|g|^2}1)\otimes (H_{\bar{f}}1)\\
  &+(H_{\bar{f}}1)\otimes (-H_{|f|^2}1+H_{|g|^2}1-H_{\bar{f}g}1-H_{\bar{f}^2}1)\\
  &+(-H_{\bar{f}\bar{g}}1-H_{|g|^2}1)\otimes (H_{\bar{g}}1)
  +(H_{\bar{g}}1)\otimes (H_{\bar{f}\bar{g}}1+H_{|g|^2}1),
\end{align*}
equivalently,
\begin{equation}
  \begin{aligned}\label{H2}
   &(H_{\bar{f}^2+\bar{f}g+|f|^2-|g|^2}1)\otimes (H_{\bar{f}}1)
   +(H_{\bar{g}}1)\otimes (H_{\bar{f}\bar{g}+|g|^2}1)\\
   &=(H_{\bar{f}}1)\otimes (H_{|f|^2-|g|^2+\bar{f}g+\bar{f}^2}1)
   +(H_{\bar{f}\bar{g}+|g|^2}1)\otimes (H_{\bar{g}}1).
 \end{aligned}
\end{equation}

In the following, we will solve the operator equation (\ref{H2}) with at most rank two into four cases:

Case 1. $H_{\bar{f}}1=0$ and $H_{\bar{g}}1=0$:

In this case, $\bar{f}$ and $\bar{g}$ are analytic, while $f$ and $g$ are analytic, then $f, g$ are constants, it follows that $\phi$ is a constant.

Case 2. $H_{\bar{f}}1=0$ and $H_{\bar{g}}1\neq 0$:

$H_{\bar{f}}1=0$ gives that $f$ is co-analytic. Since $f$ is analytic, $f$ is a constant. Equation (\ref{H2})
gives that
\[
(H_{\bar{g}}1)\otimes (H_{\bar{f}\bar{g}+|g|^2}1)= (H_{\bar{f}\bar{g}+|g|^2}1)\otimes (H_{\bar{g}}1)
\]
which means that $(H_{\bar{g}}1)\otimes (H_{\bar{f}\bar{g}+|g|^2}1)$ is a self-adjoint operator with rank one.
It follows that there exists a real constant $\lambda$ such that
$H_{\bar{f}\bar{g}+|g|^2}1=\lambda H_{\bar{g}}1$.
Then $\bar{f}\bar{g}+|g|^2-\lambda\bar{g}$ is analytic. So $\Delta(\bar{f}\bar{g}+|g|^2-\lambda\bar{g}) =0$, where $\Delta=4\frac{\partial^2}{\partial z \partial \bar{z}}$ is the Laplacian operator. Therefore, $|g'|^2=0$, which gives that $g'=0$, thus $g$ is a constant function. It is a contradiction to the condition that $H_{\bar{g}}1\neq 0$. Hence this case cannot happen.

Case 3. $H_{\bar{f}}1 \neq 0$ and $H_{\bar{g}}1= 0$:

The equation $H_{\bar{g}}1= 0$ combining with the fact that $g$ is analytic implies that $g$ is a constant.
Equation (\ref{H2}) now becomes
\[
(H_{\bar{f}^2+\bar{f}g+|f|^2-|g|^2}1)\otimes (H_{\bar{f}}1)=(H_{\bar{f}}1)\otimes [(H_{|f|^2-|g|^2+\bar{f}g+\bar{f}^2}1)
\]
which means that $(H_{\bar{f}^2+\bar{f}g+|f|^2-|g|^2}1)\otimes (H_{\bar{f}}1)$ is a rank-one self-adjoint operator.
Thus there exists some real constant $\lambda$ such that
$H_{\bar{f}^2+\bar{f}g+|f|^2-|g|^2}1=\lambda H_{\bar{f}}1$.
It follows that $\bar{f}^2+\bar{f}g+|f|^2-|g|^2-\lambda \bar{f}$ is analytic,
which implies that $\Delta(\bar{f}^2+\bar{f}g+|f|^2-|g|^2-\lambda \bar{f})=0$.
Then we have $|f'|^2=0$, which gives that $f'=0$, thus $f$ is a constant. It is a contradiction to the condition $H_{\bar{f}}1 \neq 0$. Therefore this case cannot happen.

Case 4. $H_{\bar{f}}1 \neq 0$ and $H_{\bar{g}}1 \neq 0$:

There are two possibilities in this case as follows.

Case 4.1. $H_{\bar{f}}1$ and $H_{\bar{g}}1$ are linearly dependent.

That is, there exists a constant $\lambda$ such that
$H_{\bar{f}}1=\lambda H_{\bar{g}}1$. Hence $\bar{f}-\lambda \bar{g}$ is analytic. Since both $f$ and $g$ are analytic, we have $\bar{f}-\lambda \bar{g}$ is co-analytic, which gives that $\bar{f}-\lambda \bar{g}$ is a constant.
Therefore, there exists $c \in \mathbb{C}$ such that $\bar{f}-\lambda \bar{g}=c$,
or equivalently, $f=\bar{\lambda} g+\bar{c}$.
Thus Equation (\ref{H2}) gives that
\[
(H_{\bar{\lambda}(\bar{f}^2+\bar{f}g+|f|^2-|g|^2)-(\bar{f}\bar{g}+|g|^2)}1)\otimes (H_{\bar{g}}1)
   =(H_{\bar{g}}1)\otimes (H_{\bar{\lambda}(|f|^2-|g|^2+\bar{f}g+\bar{f}^2)-(\bar{f}\bar{g}+|g|^2)}1),
\]
which means that $(H_{\bar{\lambda}(\bar{f}^2+\bar{f}g+|f|^2-|g|^2)-(\bar{f}\bar{g}+|g|^2)}1)\otimes (H_{\bar{g}}1)$
is a self-adjoint operator. Thus there exists $a \in \mathbb{R}$ such that
\[
H_{\bar{\lambda}(\bar{f}^2+\bar{f}g+|f|^2-|g|^2)-(\bar{f}\bar{g}+|g|^2)}1=a H_{\bar{g}}1.
\]
Subsequently, $\bar{\lambda}(\bar{f}^2+\bar{f}g+|f|^2-|g|^2)-(\bar{f}\bar{g}+|g|^2)-a \bar{g}$ is analytic.
Hence $\Delta (\bar{\lambda}(\bar{f}^2+\bar{f}g+|f|^2-|g|^2)-(\bar{f}\bar{g}+|g|^2)-a \bar{g})=0$,
it follows that
\[
\bar{\lambda}(\bar{f}'g'+|f'|^2-|g'|^2)-|g'|^2=0.
\]
Combining the relation $f=\bar{\lambda}g+\bar{c}$ with the above identity, we have
\[
\bar{\lambda}(\lambda |g'|^2+|\lambda|^2|g'|^2-|g'|^2)-|g'|^2=0
\]
that is,
\[
[\bar{\lambda}(\lambda +|\lambda|^2-1)-1]|g'|^2=0.
\]
Since $H_{\bar{g}}\neq 0$, $g$ is not constant, we have $|g'| \neq 0$.
It follows that
\[
\bar{\lambda}(\lambda +|\lambda|^2-1)-1=0.
\]
Thus $(\bar{\lambda}+1)(|\lambda|^2-1)=0$, which gives that $|\lambda|=1$, that is, $\lambda$ is a unimodular constant.
Let $\alpha=\bar{\lambda}$ and $\beta=\bar{c}$. Hence $\phi = \alpha g+\bar{g}+ \beta $ with $\alpha, \beta \in \mathbb{C}$ and $|\alpha|=1$.

Case 4.2. $H_{\bar{f}}1$ and $H_{\bar{g}}1$ are linearly independent.

In this case, $H_{\bar{g}}1 \neq 0$ gives that $H_{\bar{f}}1$ and $H_{\bar{f}\bar{g}+|g|^2}1$ are linearly independent.
Then Equation (\ref{H2}) gives that there exist constants $a, b, c, d$ such that
 \begin{align}
   H_{\bar{f}^2+\bar{f}g+|f|^2-|g|^2}1&= aH_{\bar{f}}1+bH_{\bar{f}\bar{g}+|g|^2}1, \label{E1}\\
   H_{\bar{g}}1&=cH_{\bar{f}}1+dH_{\bar{f}\bar{g}+|g|^2}1 \label{E2}.
 \end{align}
It follows from (\ref{E2}) that $\bar{g}-c\bar{f}-d(\bar{f}\bar{g}+|g|^2)$ is analytic.
Hence $\Delta(\bar{g}-c\bar{f}-d(\bar{f}\bar{g}+|g|^2))=0$,
which implies that $d|g'|=0$, thus $d=0$ as $g$ is not constant.
By (\ref{E2}) again, we have $H_{\bar{g}}1=cH_{\bar{f}}1$, that is, $H_{\bar{f}}1$ and $H_{\bar{g}}1$ are linearly dependent, a contradiction.
Therefore, Case 4.2 cannot happen.
This completes the proof.
\end{proof}



{\bf Acknowledgments.} This work was partially supported by the Scientific and Technological Research Program of Chongqing Municipal Education Commission (Grant No. KJQN202400829), the Chongqing Natural Science Foundation (Grant No. CSTB2025NSCQ-GPX0813) and the China Scholarship Council (Grant No. 202402380005).

\bibliographystyle{amsplain}

\end{document}